\documentclass[11pt,english]{amsart}
\usepackage{amsmath, amsthm, latexsym, amssymb}
\usepackage{amsfonts}
\usepackage{epsfig}
\input xy
\xyoption{all}
\usepackage{graphicx}
\usepackage{amssymb,tikz}
\usetikzlibrary{positioning}
\usepackage[colorlinks=false,urlbordercolor=white]{hyperref}

\newcommand{\shrinkmargins}[1]{
  \addtolength{\textheight}{#1\topmargin}
  \addtolength{\textheight}{#1\topmargin}
  \addtolength{\textwidth}{#1\oddsidemargin}
  \addtolength{\textwidth}{#1\evensidemargin}
  \addtolength{\topmargin}{-#1\topmargin}
  \addtolength{\oddsidemargin}{-#1\oddsidemargin}
 \addtolength{\evensidemargin}{-#1\evensidemargin}
  }

\shrinkmargins{.3}

\theoremstyle{plain}

\newtheorem{theorem}{Theorem}[section]
\newtheorem{corollary}[theorem]{Corollary}
\newtheorem{lemma}[theorem]{Lemma}
\newtheorem{proposition}[theorem]{Proposition}
\newtheorem{question}[theorem]{Question}

\newtheorem*{teo}{Theorem}
\newtheorem*{coro}{Corollary}
\newtheorem{definition}[theorem]{Definition}

\theoremstyle{remark}
\newtheorem{remark}[theorem]{Remark}

\theoremstyle{definition}

\def \F { \mathbb{F}}
\def \Z { \mathbb{Z}}
\def \Q { \mathbb{Q}}

\def \C { \mathbb{C}}

\def \gal { \text{Gal}}

\def \det { \text{det}}

\def\gn#1#2{{$\href{http://groupnames.org/\#?#1}{#2}$}}
\def\gn#1#2{$#2$}  
\tikzset{sgplattice/.style={inner sep=1pt,norm/.style={red!50!blue},char/.style={blue!50!black},
  lin/.style={black!50}},cnj/.style={black!50,yshift=-2.5pt,left=-1pt of #1,scale=0.5,fill=white}}

\begin{document}

\thispagestyle{empty}
\setcounter{tocdepth}{1}

\title{On a question of Perlis and Stuart regarding arithmetic equivalence. }
\author{Guillermo Mantilla-Soler}
\date{}

\maketitle

\begin{abstract}
Let $K$ be a number field. The $K$-arithmetic type of a rational prime $\ell$ is the tuple $A_{K}(\ell)=(f^{K}_{1},...,f^{K}_{g_{\ell}})$ of the residue degrees of $\ell$ in $K$, written in ascending order. A well known result of Perlis from the 70's states that two number fields have the same Dedekind zeta function if and only if for almost all primes $\ell$ the arithmetic types of $\ell$ in both fields coincide. By the end of the 90's Perlis and Stuart asked if having the same zeta function implies that for ramified primes the sum of the ramification degrees coincide. Here we study and answer their question for septic number fields.
\end{abstract}

\section{Introduction}

Two number fields are called {\it arithmetically equivalent} if they have the same Dedekind zeta function. It is of continuous interest to several authors, see for instance recent works by \cite{Cornelissen}, \cite{Prasad}, \cite{Cornelissen2} and others, to study arithmetic equivalence in number fields, and their geometric counterparts from the point of view of function fields. \\

\noindent Let $K$ be a number field and let $\ell$ be a rational prime. Recall that the {\it arithmetic type} of $\ell$ in $K$  is the ordered tuple \[A_{K}(\ell):=(f^{K}_{1},...,f^{K}_{g^{K}})\] where the $f^{K}_{i}$'s are the residue degrees of $\ell$ in $K$ and \[f^{K}_{1} \leq ...\leq f^{K}_{g^{K}}.\] Let $e^{K}_{i}$ be the ramification degree of $\ell$ corresponding to the residue degree $f^{K}_{i}$. In the early 70's Perlis showed that the Dedekind zeta function of a number field is completely determined by the residue degrees over every rational prime. Further, he gave a group theoretic characterization for the equivalence to occur:\\

\noindent Suppose that $K,K'$ are two number fields and let $N$ be the compositum of their Galois closures. Let $G={\rm Gal}(N/\Q)$, and let $H$, $H_{1}$ be the corresponding subgroups of $K$ and $K'$ via Galois

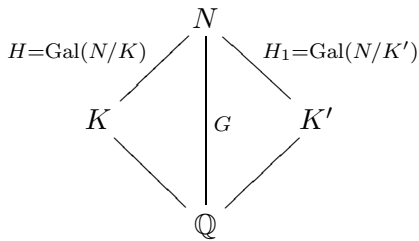
\[\xymatrix{  
     & N \ar@{-}[dd]^{G}\ar@{-}[ld]_{H={\rm Gal}(N/K)}\ar@{-}[rd]^{H_{1}={\rm Gal}(N/K')}   & \\
             K\ar@{-}[rd] & &  K' \ar@{-}[ld] \\
            & \Q &   }\]
\begin{definition}
We say that  $K$ and $K'$ are quasi-conjugate if the groups $H$ and $H_{1}$ are quasi-conjugate in $G$ i.e., if for very conjugacy class $C$ of $G$ we have that \[\#(C \cap H) =\#(C \cap H_{1}).\]

\end{definition}

\begin{remark}\label{ElCuasiConjugados}
Notice that if $H$ and $H_{1}$ are conjugate subgroups in $G$ then they are quasi-conjugate. Also, since conjugacy classes are a partition of $G$ we have that $\#H=\#H_{1}$ whenever $H$ and $H_{1}$ are quasi-conjugate
\end{remark}

\begin{theorem}[\cite{Perlis1}]\label{EldePerlis} Let $K,K'$ be two number fields. Then, the following are equivalent:  

\begin{itemize}

\item[(a)] The fields $K,K'$ are arithmetically equivalent. 

\item[(b)] For almost every prime $\ell$ the arithmetic types of $\ell$ in $K$ and $K'$ are the same.

\item[(c)] The fields $K,K'$ are quasi-conjugate.
 
\end{itemize}
\end{theorem}

One useful application of the above theorem is that of easily checking when a field $K$ is {\it arithmetically solitary}, i.e., there is no a field $K'$ non-isomorphic and arithmetically equivalent to $K$. For example, using this and some group theory, Perlis \cite{Perlis1} has proved that if the degree of $K$ is at most 6 then it is arithmetically solitary. Moreover if $K$ is septic and it is not arithmetically solitary then $K$ is a ${\rm PSL_{2}}(\F_{7})$ septic field i.e., the Galois group of its Galois closure is the simple group of order $168$ (see \cite{Klingen1} or \cite{Praeger}).

\subsection{The question and its answer}

In the late 90's Perlis and Stuart gave a new surprising characterization for arithmetic equivalence; They showed that it is enough to know the length of the arithmetic types, at almost every prime, to know the Dedekind zeta function. Explicitly:

\begin{theorem}[Perlis, Stuart \cite{PerStu}]
 Let $K,K'$ be number fields. Then, $K,K'$ are arithmetically equivalent if and only if for almost every prime $\ell$ the number of prime factors lying over $\ell$ in $O_{K}$ and $O_{K'}$ is the same.

\end{theorem} 

Perlis and Stuart asked if not only the residue degrees or the number of prime factors are determined by the zeta function, but if the ramification degrees are determined as well. This is not the case, as shown by Perlis. However, Perlis and Stuart pointed out that it was not known if the sum of the ramification degrees can differ. Suppose that $K,K'$ are arithmetically equivalent  number fields and let $\ell$ be a rational prime. Let  $A_{K}(\ell)=(f_{1},...,f_{g})=A_{K'}(\ell)$ be the common arithmetic type tuples for $\ell$. Let $(e^{K}_{1},...,e^{K}_{g})$ and $(e^{K'}_{1},...,e^{K'}_{g})$ be the corresponding tuples of ramification degrees. Based on their work on split and arithmetic equivalence, and the examples they studied, Perlis and Stuart ended their paper with the following question:

\begin{question}[Perlis, Stuart \cite{PerStu}]\label{LaPregunta}

Does it  follow that the sum of the ramification degrees is the same for all prime $\ell$?

\[e^{K}_{1}+...+e^{K}_{g}=e^{K'}_{1}+...+e^{K'}_{g}\]

\end{question}

In principle the only obvious restriction on the ramification degrees is given by the following:

\[f_{1}e^{K}_{1}+...+f_{g}e^{K}_{g}=f_{1}e^{K'}_{1}+...+f_{g}e^{K'}_{g}=[K:\Q].\]

\subsubsection{Septic fields}

Since degree less than $7$ number fields are arithmetically solitary the first interesting case of study for Question \ref{LaPregunta} is that of septic number fields. As it turns out, in degree $7$,  under certain restrictions on the ramification types Question \ref{LaPregunta} is positively answered:

\begin{teo}[cf. Theorem \ref{Cases}] Let $K$ be a degree $7$ number field, and let $\ell$ be a rational prime. Suppose that the arithmetic type of $\ell$ in $K$ does not belong to \[\{(1,3), (1,1,2), (1,1,1,2)\}.\] Then for any $K'$ arithmetically equivalent to $K$  \[e^{K}_{1}+...+e^{K}_{g}=e^{K'}_{1}+...+e^{K'}_{g}.\]
\end{teo} 

From this and from the fact that septic fields that are non arithmetically solitary have a Galois closure with simple Galois group, we obtain:

\begin{coro}[cf. Corollary \ref{ElCoro}]
Let $K, K'$ be degree $7$ arithmetically equivalent number fields. Suppose $\ell$ is a rational prime which is not wildly ramified in either $K$ or $K'$, and let $v_{\ell}$ the usual $\ell$-adic valuation. If $\displaystyle e^{K}_{1}+...+e^{K}_{g} \neq e^{K'}_{1}+...+e^{K'}_{g}$ then $v_{\ell}({\rm disc}(O_{K})) \in \{2, 4\}$.%
\end{coro}

\subsection{What about general non arithmetically solitary septic fields?}

Based on the above results we use an algorithm that searches for possible examples giving a negative answer to Question \ref{LaPregunta}. First we search for pairs of septic fields with same  discriminant, and signature, with not too many ramified primes; actually to make things easier we start with only two prime factors. Moreover, using the remark after Corollary \ref{ElCoro}, we  take the prime $\ell=2$ as one of the two primes. Similarly, thanks to Corollary \ref{ElCoro}, we know the valuation of the discriminant at the other prime divisor as well. Among the candidates found, we select the ones such that their ramified primes have arithmetic type belonging to the list appearing in Theorem \ref{Cases}. From those, we take the ${\rm PSL}_{2}(\F_{7})$ fields and see if there is any couple of arithmetically equivalent fields for which the sum of ramification degrees differ. More explicitly:

\subsubsection{Algorithm}

The input is a list of septic fields up to some discriminant bound, and the output is either a list either empty or containing pairs of examples, within the discriminant bound, giving a negative answer to Question \ref{LaPregunta}.

\begin{itemize}
    \item[(i)] Look in the list for number fields with discriminant of the form $2^{2a}p^{2b}$ where $a \in \{3,4\}$ and $b \in \{1,2\}$.\\
    
    \item[(ii)] Select pairs of fields from step (i) that have equal signature and ${\rm discriminant}$ and such that their ramification types, at $2$ or $p$, belong to  the list appearing in Theorem \ref{Cases}. \\
    
    \item[(iii)] From (ii) select the ones that have Galois group ${\rm PSL}_{2}(\F_{7})$.\\
    
    \item[(iv)] Verify, using Theorem \ref{DisjointPrime}, whether or not the fields obtained in (iii) are arithmetically equivalent.\\
    
    \item[(v)] From each pair of arithmetically equivalent fields obtained check whether or not there are pairs for which the sum of ramification degrees, at the ramified primes, are different.\\
    
\end{itemize}

Using the algorithm described above with John Jones' data base of number fields, and writing some MAGMA code, we found out that Theorem \ref{Cases} is optimal for getting a positive answer to Question \ref{LaPregunta}; in other words:

\begin{teo}[cf. Theorem \ref{Larespuesta}]
For each tuple $\mathcal{F} \in \{(1,3), (1,1,2), (1,1,1,2)\}$ there are examples of pairs $(K,K')$ of non-isomorphic arithmetically equivalent number fields, and a prime $\ell$, with common arithmetic type $\mathcal{F}$ in $K$ and $K'$ such that \[e^{K}_{1}+...+e^{K}_{g} \neq e^{K'}_{1}+...+e^{K'}_{g}.\] 
\end{teo}

\subsubsection{Overview of the contents}

In Section \S \ref{LasGalRep} we recall most of the standard facts of Arithmetic equivalence from the point of Galois representations. Nothing in this section is new and it is well known to experts but we could not find a suitable reference that contains these results in the context of $G_{\Q}$ representations. For example, even though that from the point of view of Galois representations Theorem \ref{main} is elementary we have not found a presentation of that result in such a natural form. In Section \S \ref{LasPruebas} we give proofs of our main results and exhibit examples, found following the algorithm described above, that give a negative answer to Question \ref{LaPregunta}.

\section{Arithmetic equivalence via Galois representations}\label{LasGalRep}

Suppose that the Dedekind zeta function of a number field $K$ is written as \[\zeta_{K}(s):=\sum_{n=0}^{\infty} \frac{a_{n}(K)}{n^{s}}.\] Then interpreting the zeta function as a counting function it should be true, as in the case of Tate's isogeny theorem, that $\zeta_{K}(s)$ is completely determined by the values $a_{\ell}(K)$ at primes $\ell$. Since $a_{\ell}(K)$ is equal to the number $1$'s appearing in the tuple $A_{K}(\ell)$ knowing the values $a_{\ell}(K)$ is a priori weaker than knowing the the arithmetic type of $\ell$ in $K$. However, as suggested above, the knowledge of the $a_{\ell}(K)$ for almost all $\ell$ indeed determines the function $\zeta_{K}(s)$. In this section, using the rudiments of Galois representations,  we briefly recall how these results can be obtained.\\

\noindent Let $K$ be a degree $n$ number field, and let us denote by $\widetilde{K}$ its Galois closure over $\Q$. We start by recalling the construction of an $n$-dimensional complex Galois representation $\rho_{K}$ of the absolute Galois group $G_{\Q}$ such that the Artin $L$-function associated to $\rho_{K}$ is $\zeta_{K}(s)$. Let ${\rm Emb}(K)$ be the set of its complex embeddings of $K$. The absolute Galois group $G_{\Q}$ acts continuously on ${\rm Emb}(K)$ via composition. The continuity follows since the kernel of the action is the open group $G_{\widetilde{K}}$. Since   $n=\#{\rm Emb}(K)$ the above gives a continuous permutation representation   $G_{\Q}: \pi_{K} \to S_{n}$, which by composition with the permutation representation $\iota_{n}: S_{n} \to \rm{GL}_{n}(\C)$ produces an $n$-dimensional complex representation \[\rho_{K} : G_{\Q} \to \rm{GL}_{n}(\C).\]

\begin{definition}\label{TateModuleNumberField}
Let $K$ be  a number field. The continuous $\C[G_{\Q}]$-module $T_{K}$ is the $G_{\Q}$-module attached to the representation $\rho_{K}$. In other words, $\displaystyle T_{K}:=\bigoplus_{ \sigma \in {\rm Emb} } \C\sigma$ with the action of $G_{\Q}$ in each element of the basis given by composition.
\end{definition}

The relevance of this representation to our purposes is that the Artin formalism gives us the following:

\begin{proposition}\label{ZetaLArtin}
Let $K$ be a number field and let us denote by $L(\rho, s)$ the Artin L-function attached to a representation $\rho$. Then $\displaystyle L(\rho_{K}, s)= \zeta_{K}(s).$
\end{proposition}

\begin{proof}

By Galois correspondence $\rho_{K}$ factorizes through  ${\rm Res}_{\widetilde{K}} ^{\overline{\Q}}(\rho_{K}):  {\rm Gal}(\widetilde{K}/\Q) \to \rm{GL}_{n}(\C)$. Again, by basic Galois theory, the action of ${\rm Gal}(\widetilde{K}/\Q)$ in ${\rm Emb}(K)$ is isomorphic to the permutation representation of  ${\rm Gal}(\widetilde{K}/\Q)$ in the set of cosets  ${\rm Gal}(\widetilde{K}/\Q)/ {\rm Gal}(\widetilde{K}/K)$. Hence, 
$\displaystyle {\rm Res}_{\widetilde{K}} ^{\overline{\Q}}(\rho_{K}) \cong  {\rm Ind}_{{\rm Gal}(\widetilde{K}/K)}^{{\rm Gal}(\widetilde{K}/\Q)} 1_{{\rm Gal}(\widetilde{K}/K)}.$ Thanks to Artin's formalism 
\[ L(\rho_{K}, s)=L\left({\rm Res}_{\widetilde{K}} ^{\overline{\Q}}(\rho_{K}), s\right) =L\left( {\rm Ind}_{{\rm Gal}(\widetilde{K}/K)}^{{\rm Gal}(\widetilde{K}/\Q)} 1_{{\rm Gal}(\widetilde{K}/K)} , s\right)=L(1_{{\rm Gal}(\widetilde{K}/K)} ,s)= \zeta_{K}(s).\] \end{proof}

Since the Dedekind zeta function is an Artin $L$-function then its prime terms correspond to traces of Frobenius elements:

\begin{corollary}\label{Frobenius}

Let $K$ be a number field and $\ell$ be a prime unramified\footnote{This is the same as being unramified in $K$ since the conductor of $\rho_{K}$ is the discriminant of $K$.} under $\rho_{K}$. Let $Frob_{\ell}$ be the conjugacy class of the element Frobenius at $\ell$. Then,  \[{\rm Trace}(\rho_{K}(Frob_{\ell})) =a_{\ell}(K).\]
\end{corollary}

Proposition \ref{ZetaLArtin} gives not only a simple way to express the trace of Frobenius but it also gives a useful generalization of the above corollary to calculate its characteristic polynomial $\displaystyle \det(X-\rho_{K}({\rm Frob}_{\ell})).$

\begin{lemma}
Let $K$ be a number field and $\ell$ be a prime, unramified in $K$, and let $(f_{1},...,f_{g})$ be the arithmetic type of $\ell$ in $K$. Then, 
\[{\rm det}(X-\rho_{K}({\rm Frob}_{\ell}))=\prod_{i=1}^{g}(X^{f_{i}}-1).\]

\end{lemma}

\begin{proof}
Let $B_{1},...,B_{g}$ be the primes in $O_{K}$ lying over the prime $\ell$. Then, the $\ell$-factor in the Euler product for $\zeta_{K}(s)$ is given by \[\prod_{i=1}^{g}(1-|| B_{i} || ^{-s})^{-1}=\prod_{i=1}^{g}(1-\ell^{-sf_{i}})^{-1}.\] On the other hand, since $\zeta_{K}(s)$ is also the Artin L-function of the representation $\rho_{K}$, the $\ell$-factor above is also equal to $\det(I-\ell^{-s}\rho_{K}({\rm Frob}_{\ell}))^{-1}$. The result follows from substituting $\ell^{-s}$ by $X$.

\end{proof}

\subsubsection{An analogy with the isogeny theorem}\label{IsogenyTheorem}
The zeta function $\zeta_{K}(s)$ is the Artin $L$-function of the trivial representation of $\gal(\overline{K}/K)$, however knowing this is not very useful in our context since for two different number fields we would get representations from different groups. By looking at a Galois representation of $G_{\Q}$ for which $\zeta_{K}(s)$ is its Artin $L$-function one can actually obtain results about the number field in question. This, as straightforward as it is, gives a simpler characterization for arithmetic equivalence which is completely reminiscent of Tate's isogeny theorem on rational elliptic curves, where the $G_{\Q}-$module $\displaystyle T_{K} $ plays the role of Tate's module.

\begin{theorem}\label{main}
Let $K$, $K_{1}$ be two number fields. The following are equivalent:

\begin{itemize}

\item[${\rm (i)}$] There is a $\C$-isomorphism of $\displaystyle T_{K} \cong T_{K_{1}} $ as $G_{\Q}$-modules.

\item[${\rm (ii)}$] $\displaystyle \zeta_{K}(s)= \zeta_{K_{1}}(s).$

\item[${\rm (iii)}$] For almost all primes	 $\ell$, $\displaystyle a_{\ell}(K)=a_{\ell}(K_{1}).$

\item[${\rm (iv)}$] For almost all primes $\ell$,  $\displaystyle \#{\rm Spec}(O_{K})(\F_{\ell})=\#{\rm Spec}(O_{K_{1}})(\F_{\ell})$  i.e., $K$ and $K_{1}$ have the same number of $\F_{\ell}$ points.

\end{itemize}

\end{theorem} 

\begin{proof} We first make the following observations:

\begin{itemize}

\item[$\bullet$] By the uniqueness theorem for Dirichlet series we have that   $\zeta_{K}(s)=\zeta_{K_{1}}(s)$ implies that $a_{\ell}(K)=a_{\ell}(K_{1})$ for all prime $\ell$. 

\item [$\bullet$] Thanks to Corollary \ref{Frobenius} $\ a_{\ell}(K)=a_{\ell}(K_{1})$ for all primes $\ell$  implies that \[\displaystyle   {\rm Trace}(\rho_{K}(Frob_{\ell}))={\rm Trace}(\rho_{K_{1}}(Frob_{\ell}))\] for almost all primes $\ell$. 

\item [$\bullet$] By Chebotarev's density theorem  $\displaystyle   {\rm Trace}(\rho_{K}(Frob_{\ell}))={\rm Trace}(\rho_{K_{1}}(Frob_{\ell}))$ for almost all primes $\ell$ is equivalent to $\displaystyle {\rm Trace}(\rho_{K}(g))={\rm Trace}(\rho_{K_{1}}(g))$ for all $g \in G_{\Q}$.

\item [$\bullet$] Since Artin representations have finite images, the fact that $\displaystyle {\rm Trace}(\rho_{K}(g))={\rm Trace}(\rho_{K_{1}}(g))$ for all $g \in G_{\Q}$ implies that $\rho_{K}$ and $\rho_{K_{1}}$  are isomorphic representations.

\item [$\bullet$] If the representations $\rho_{K}$ and $\rho_{K_{1}}$  are isomorphic then, thanks to Proposition \ref{ZetaLArtin},  $\zeta_{K}(s)=\zeta_{K_{1}}(s)$. 
\end{itemize} The above argument shows the equivalence between (i), (ii) and (iii). Suppose that $K$ is defined by a monic polynomial $p(x) \in \Z[x]$, and suppose that $\ell \nmid {\rm disc}(p)$. The equivalence with (iv) follows since \[ \#{\rm Spec}(O_{K})(\F_{\ell}) =\{ \alpha \in \F_{\ell} \mid f(\alpha) =0 \}=\#\{f  \in A_{K}(\ell) \mid f =1\}=a_{\ell}(K).\]
\end{proof}

\begin{remark}
Conditions (iii) or (iv) in Theorem \ref{main} are a priori weaker conditions for arithmetic equivalence than the ones given by Perlis and others; even though condition (iii) seems quite natural as an equivalence for the equality between Dedekind zeta functions it's not normally mentioned in this form. Here is usual formulation of this equivalence:
\end{remark}

\begin{corollary}\label{ElResultado}

Let $K, K_{1}$ be two number fields. Then $K$ and $K_{1}$ are arithmetically equivalent if and only if for almost all rational primes $\ell$ 
\[\#\{f  \in A_{K}(\ell) \mid f =1\} =\#\{f  \in A_{K_{1}}(\ell) \mid f =1\}.\]

\end{corollary}

\begin{proof}
Since  \[a_{\ell}(K)=  \#\{\mathcal{B} \in {\rm Max}(O_K) \mid  [O_{K}: \mathcal{B}] =\ell \}  = \#\{f  \in A_{K}(\ell) \mid f =1\}\] the result follows from Theorem \ref{main} \end{proof}

\subsubsection{Invariants under arithmetic equivalence}\label{InvAritEquiv}

Some of the invariants determined by arithmetic equivalence are the degree, the discriminant, the signature, the Galois closure, the roots of unity and the unit group (see for example \cite[III, \S1, Theorem 1.1]{Klingen}). All of them can be easily explained by the Galois representation $\rho_{K}$. For instance the degree is the dimension of $\rho_{K}$, the number of real embeddings of $K$ is ${\rm Trace}(\rho_{K}({\rm complex \ conjugation}))$, the discriminant is equal to the conductor of $\rho_{K}$ (see \cite[VI, \S3, Corollary 1]{Serre2}), etc.\\  

\noindent Other invariants determined under arithmetic equivalence are the rational trace form, or under some ramification conditions the integral trace form. To see how those can be deduced also from $\rho_{K}$ the reader can see \cite{MantillaS} or \cite{Perlis}.\\

\paragraph{{\it Quasi-conjugate subgroups}}

The classic group theoretical characterization of Perlis \cite{Perlis1} and Gassmann \cite{Gassmann} for arithmetic equivalence, see Theorem \ref{EldePerlis} (c)-(a), can be made quite clear from the point of view of the representation $\rho_{K}$. More precisely:

\begin{corollary}\label{ZetaQuasiConj} Let $K, K_{1}$ be number fields and let $N$ be a Galois number field such that $KK_{1} \subseteq N.$ Let $G:={\rm Gal}(N/\Q)$, and $H$, $H_{1}$ be the sub-groups of $G$ corresponding to $K$ and $K_{1}$ via Galois correspondence. Then $\zeta_{K}(s)=\zeta_{K_{1}}(s)$ if and only if $H$ and $H_{1}$ are quasi-conjugate in $G$.
\end{corollary}

\begin{proof}
 Since $N$ is Galois over $\Q$ it contains $\widetilde{K}$ and $\widetilde{K}_{1}$. Using Artin's formalism as in the proof of Proposition \ref{ZetaLArtin} we see that $\zeta_{K}(s)=L(\rho_{K}, s)=L\left({\rm Res}_{N} ^{\overline{\Q}}(\rho_{K}), s\right) =L \left( {\rm Ind}_{H}^{G} 1_{H} , s \right)$, resp. the analog statement for $H_{1}$. Therefore, thanks to Lemma \ref{LemmaQuasi}, if $H$ and $H_{1}$ are quasi conjugate then $\zeta_{K}(s)=\zeta_{K_{1}}(s)$. On the other hand if $\zeta_{K}(s)=\zeta_{K_{1}}(s)$ then we see, from Theorem \ref{main}, that the representations $\rho_{K}$ and $\rho_{K_{1}}$ are isomorphic. Restricting this isomorphism to $G_{N}={\rm Gal}(\overline{\Q}/N)$ it follows, from Lemma \ref{LemmaQuasi}, that $H$ and $H_{1}$ are quasi-conjugate since $\displaystyle {\rm Res}_{N} ^{\overline{\Q}}(\rho_{K}) \cong  {\rm Ind}_{H}^{G} 1_{H}$ and $\displaystyle {\rm Res}_{N} ^{\overline{\Q}}(\rho_{K}) \cong  {\rm Ind}_{H_{1}}^{G} 1_{H_{1}}$
\end{proof}

\begin{lemma}\label{LemmaQuasi}

Let $G$ be a finite group and let $H$ and $H_{1}$ two subgroups. Then $H$ and $H_{1}$ are quasi-conjugate if and only if $\displaystyle {\rm Ind}_{H}^{G} 1_{H} \cong {\rm Ind}_{H_{1}}^{G} 1_{H_{1}}.$
\end{lemma}

\begin{proof}
Let $\chi_{H}$ be the character afforded by the representation ${\rm Ind}_{H}^{G} 1_{H}$ and let $C$ be a conjugacy class in $G$. A calculation shows that  \[\chi_{H}(C)=\frac{\#(C \cap H)\#G}{\#C\#H}.\] Taking the trivial conjugacy class we see that the order of $H$ is determined by the representation, hence the result follows the above equality and from the definition of quasi-conjugate subgroups(see Remark \ref{ElCuasiConjugados}) .
\end{proof}

\section{Proofs of our results}\label{LasPruebas}

Let $K$ be a number field with maximal order $O_{K}$ and let $\ell$ be a rational prime. Recall that the arithmetic type of $\ell$ in $K$ is the tuple $A_{K}(\ell)=(f^{K}_{1},...,f^{K}_{g_{\ell}})$ written in ascending order where $g_{\ell}^{K}$ is the number of prime factors of $\ell$ in $O_{K}$ and the $f_{i}'s$ are the residue degrees of $\ell$. Let $e^{K}_{i}$ be the ramification degree corresponding to the residue degree $f^{K}_{i}$. We call {\it the factorization type} of a prime $\ell$ the two element ordered set \[\{(f_1,...,f_g), (e_1,...,e_g)\} \  \] where the first tuple is the arithmetic type and the second is the tuple of ramification indices corresponding to each residue degrees. In this section we study the possible factorization types for primes in septic number fields.

\subsection{${\rm PSL}_{2}(\F_{7})$ number fields.}

Since septic number fields that are not arithmetically solitary are ${\rm PSL}_{2}(\F_{7})$ fields it is of interest for us to study what happens with ramification in  ${\rm PSL}_{2}(\F_{7})$ number fields.

\begin{proposition}\label{Esporadicos}

Let $K$ be a septic number field with Galois closure having Galois group isomorphic to ${\rm PSL}_{2}(\F_{7})$. If $\ell$ is a rational prime, then its factorization type is not equal to $\mathcal{T}:=\{(1,2,2), (3,1,1)\}$
\end{proposition}

\begin{proof}
Let $\ell$ be a prime and suppose that its factorization type is equal to $\mathcal{T}$. Let $P_{1}, P_{2}$ and $P_{3}$ be the primes in $O_{K}$ lying over $\ell$ and such that \[\ell O_{K} = P_{1}^{3} P_{2}P_{3}.\] Let $L$ be the Galois closure of $K$ over $\Q$. For each $i =1,2,3$ let $e_{i}$, $f_{i}$ and $g_{i}$ be respectively the ramification index, inertia degree and number of prime factors in $O_{L}$ of the prime $P_{i}$. By the hypothesis on $K$ and $\ell$ we have that \begin{align*} e_{i}f_{i}g_{i}=24\end{align*} for all $i =1,2,3$. Moreover, if $e$, $f$ and $g$ are the respective values for the extension $L/\Q$ and the prime $\ell$ then \begin{align*}
e&= 3e_{1}=e_{2}=e_{3}\\
   f    &=f_{1}=2f_{2}=2f_{3}\\
    g   &=g_{1}+g_{2}+g_{3}
\end{align*}
It follows from the above equations that $g_{2}=g_{3}$ and that $2g_{1}=3g_{2}$. In particular, $g=\frac{7}{2}g_{2}$ is a multiple of $7$. Therefore $ef$, which is the order of a decomposition group over $\ell$ in the extension $L/\Q$, must be a divisor of $24$. Since $3 \mid e$ and $2 \mid f$ we have that $ef \in \{6, 12, 24\}$. Now, let $D_{i} \leq {\rm Gal}(L/K)$ be a decomposition subgroup of  for the prime $P_{i}$. Since decomposition groups can be extended there is, for the prime  $\ell$, a decomposition subgroup  $E_{i} \leq {\rm Gal}(L/\Q) \cong {\rm PSL}_{2}(\F_{7}) $  such that $E_{i} \cap {\rm Gal}(L/K) =D_{i}$. Thus the group $E_{i}$, which has order $ef$, has for each $i$ a subgroup of order $e_{i}f_{i}$. We recall the lattice of sub-groups of ${\rm PSL}_{2}(\F_{7})$, modulo conjugacy:
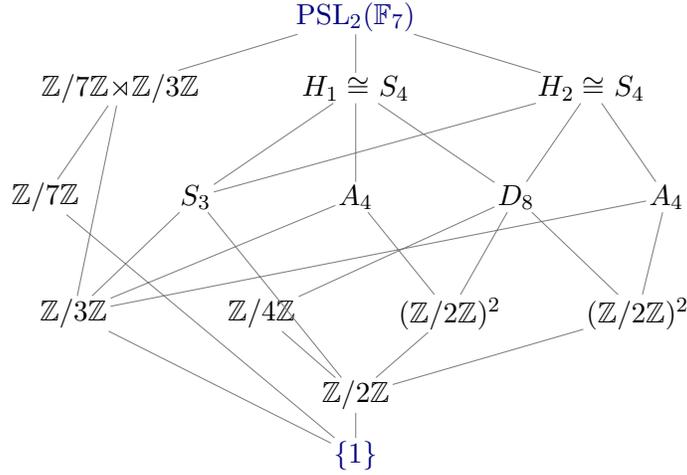
\[
\begin{tikzpicture}[scale=1.0,sgplattice]
  \node[char] at (4.25,0) (1) {\gn{C1}{\{1\}}};
  \node at (4.25,0.803) (2) {\gn{C2}{\Z/2\Z}};
  \node at (0.5,1.89) (3) {\gn{C3}{\Z/3\Z}};
  \node at (0.125,3.44) (4) {\gn{C7}{\Z/7\Z}};
  \node at (8,1.89) (5) {\gn{C2^2}{(\Z/2\Z)^2}};
  \node at (5.5,1.89) (6) {\gn{C2^2}{(\Z/2\Z)^2}};
  \node at (3,1.89) (7) {\gn{C4}{\Z/4\Z}};
  \node at (2.12,3.44) (8) {\gn{S3}{S_3}};
  \node at (1.12,4.89) (9) {\gn{C7:C3}{\Z/7\Z{\rtimes}\Z/3\Z}};
  \node at (6.38,3.44) (10) {\gn{D4}{D_8}};
  \node at (4.25,3.44) (11) {\gn{A4}{A_4}};
  \node at (8.38,3.44) (12) {\gn{A4}{A_4}};
  \node at (7.38,4.89) (13) {$H_{2} \cong$ \gn{S4}{S_4}};
  \node at (4.25,4.89) (14) {$H_{1} \cong$ \gn{S4}{S_4}};
  \node[char] at (4.25,5.84) (15) {\gn{PSL(2,7)}{{\rm PSL}_2({\mathbb F}_7)}};
  \draw[lin] (1)--(2) (1)--(3) (1)--(4) (2)--(5) (2)--(6) (2)--(7) (2)--(8)
     (3)--(8) (3)--(9) (4)--(9) (5)--(10) (6)--(10) (7)--(10) (3)--(11)
     (6)--(11) (3)--(12) (5)--(12) (12)--(13) (8)--(13) (10)--(13) (8)--(14)
     (10)--(14) (11)--(14) (13)--(15) (14)--(15) (9)--(15);
  \node[cnj=2] {};
  \node[cnj=3] {};
  \node[cnj=4] {};
  \node[cnj=5] {};
  \node[cnj=6] {};
  \node[cnj=7] {};
  \node[cnj=8] {};
  \node[cnj=9] {};
  \node[cnj=10] {};
  \node[cnj=11] {};
  \node[cnj=12] {};
  \node[cnj=13] {};
  \node[cnj=14] {};
\end{tikzpicture}
\]

We show separately that neither of the possibilities, $\{6,12,24\}$, can't occur as the value of $ef$: \\

\begin{itemize}

\item $ef=12.$ It follows from the equations above  that $e_{2}f_{2}=6$. Hence, the order 12 group $E_{i}$ has an order $6$ subgroup. This is a contradiction since $A_{4}$ has no subgroups of order $6$ and, see diagram above, every subgroup of ${\rm PSL_{2}}(\F_{7})$ of order $12$ is isomorphic to $A_{12}$.\\

\item $ef=24.$ It follows from the equations above  that $g_{1}=3$, $e_{2}f_{2}=12$ and $g_{2}=g_{3}=2$. Since none of the $g_{i}$'s is equal to $1$ the group ${\rm Gal}(L/K)$ can't be conjugate to a decomposition group over $\ell$; otherwise $K$ would be the fixed field of a decomposition group of a prime $\mathcal{B}$ in $O_{L}$ lying over $\ell$. In particular the prime $P:=\mathcal{B} \cap O_{K}$ would have only one prime factor in $O_{K}$, which is a contradiction since $P=P_{i}$ for some $i$. Therefore we may assume that $E_{i}$ is not conjugate to ${\rm Gal}(L/K)$. Looking at the lattice of subgroups of ${\rm PSL_{2}}(\F_{7})$ we see that no subgroup of order $12$ is the intersection of two non-conjugate subgroups of order $24$. \\

\item $ef=6.$ It follows from the equations above that $e_{1}f_{1}=3$ and $e_{2}f_{2}=2$. From the lattice of subgroups we see that the intersection of a group of order 24 with one of order 6 can't have order $2$.
\end{itemize} 

\end{proof}

\begin{proposition}\label{Difi}

Let $K$ be a septic number field with Galois closure having Galois group isomorphic to ${\rm PSL}_{2}(\F_{7})$. If $\ell$ is a rational prime, then its factorization type is not equal to either   $\{(1,2), (3,2)\}$ nor $\{(1,2), (5,1)\}$.

\end{proposition}
 
\begin{proof}
The case $\{(1,2), (5,1)\}$ is clear since $5 \nmid 168.$ Let $\ell$ be a prime and suppose that its factorization type is equal to $\{(1,2), (3,2)\}$. Let $P_{1}$ and $P_{2}$ be the primes in $O_{K}$ lying over $\ell$ and such that \[\ell O_{K} = P_{1}^{3} P_{2}^{2}.\] Let $L$ be the Galois closure of $K$ over $\Q$. For each $i =1,2$ let $e_{i}$, $f_{i}$ and $g_{i}$ be respectively the ramification index, inertia degree and number of prime factors in $O_{L}$ of the prime $P_{i}$. By the hypothesis on $K$ and $\ell$ we have that \begin{align*} e_{i}f_{i}g_{i}=24\end{align*} for all $i =1,2$. Moreover, if $e$, $f$ and $g$ are the respective values for the extension $L/\Q$ and the prime $\ell$ then \begin{align*}
e&= 3e_{1}=2e_{2}\\
   f    &=f_{1}=2f_{2}\\
    g   &=g_{1}+g_{2}.
\end{align*}
It follows from the above equations that $4g_{1}=3g_{2}$. In particular, $g=\frac{7}{4}g_{2}$ is a multiple of $7$. Therefore $ef$, which is the order of a decomposition group over $\ell$ in the extension $L/\Q$, must be a divisor of $24$. Since $6 \mid e$ and $2 \mid f$ we have that $ef \in \{12, 24\}$. As before, we deal with each possible value of $ef$ separately: \\

\begin{itemize}

\item $ef=12.$ It follows from the equations above that $e=6$. Since the inertia subgroup at $\ell$ has order $e=6$ and  $A_{4}$ has no subgroups of order $6$ this case can't happen. \\

\item $ef=24.$ From the equations we have that $e_{1}f_{1}=8$ and $e_{2}f_{2}=6$. Furthermore, either $e=12$ or $e=6$. In the former case $e_{2}=6$ and then we would have a group of order $12$, inertia, with a subgroup of order $6$ which is impossible in ${\rm PSL}_{2}(\F_{7})$. In the latter case $e_{1}=2$ and $f_{1}=4$, therefore $D_{1}$ is an order $8$ group with a cyclic quotient of order $4$; this is a contradiction since ${\rm PSL}_{2}(\F_{7})$ has no such a subgroup.

\end{itemize}
\end{proof}

\begin{remark}\label{ElRemarque}
Similarly to Proposition \ref{Difi} there is no ${\rm PSL}_{2}(\F_{7})$ septic field $K$ and a prime $\ell$ such that its factorization type is $\{(1,2), (1,3)\}$. This, together with the last proposition, shows that in $K$ the arithmetic type of a prime $\ell$ can not ever be $(1,2)$. We do not prove this here since we already have the necessary material to prove one or our main results:
\end{remark}

\begin{theorem}\label{Cases} Let $K$ be a degree $7$ number field, and let $\ell$ be a rational prime. Suppose that the arithmetic type of $\ell$ in $K$ does not belong to \[\{(1,3), (1,1,2), (1,1,1,2)\}.\] Then for any $K'$ arithmetically equivalent to $K$ \[e^{K}_{1}+...+e^{K}_{g}=e^{K'}_{1}+...+e^{K'}_{g}.\]

\end{theorem}

\begin{proof}

Let $(f_{1},...,f_{g})$ be the arithmetic type of $\ell$ in either field. The arithmetic type together with the ramification degrees gives a partition of $7$, $\displaystyle f_{1}e_{1}+...+f_{g}e_{g}=7$, so we analyze each partition of $7$ of size $g$ and see what are the possibilities for sum of the ramification degrees given the knowledge of the arithmetic type.

\begin{itemize}

\item $g=1.$

In this case the ramification degree is completely determined by the value of the residue degree.

\item $g=7.$

\begin{itemize}

\item[$\cdot$] $1+1+1+1+1+1+1$. In this case all the ramification degrees are equal to $1$.

\end{itemize}

\item $g=6.$

\begin{itemize}

\item[$\cdot$] $1+1+1+1+1+2$. In this case all the five ramification degrees are $1$ and the last one is completely determined by its corresponding residue degree.

\end{itemize}

\item $g=5.$ In principle for this case one could have different ramification degrees for same arithmetic types, however the sum of the ramification degrees is the same:

\begin{itemize}

\item[$\cdot$] $1+1+1+1+3$. In this case four residue degrees are $1$, and so they are their corresponding ramification degrees. In either case for this partition the last residue degree determine the last ramification degree. Moreover if the last residue degree is $3$ the sum of the ramification degrees is  $5$ otherwise it's $7$.

\item[$\cdot$] $1+1+1+2+2$. In this case three residue degrees are $1$, and so they are their corresponding ramification degrees. For this partition the knowledge of the arithmetic type determines the remaining ramification degrees (they are $1$ or $2$). On the other hand the only way in which this partition could have the same arithmetic type of the above partition is that all the residue degrees are equal to $1$(four of them are already $1$ and the remaining one must divide $2$ and $3$). In such a case the remaining ramification degrees are equal to $2$ and the sum of the ramification degrees is $7$ which coincides with the previous case. \\
\end{itemize}

\end{itemize}

\noindent For the partitions of size $4,3,2$ we list all the possible candidates to factorization type: \[\{(f_1,...,f_g), (e_1,...,e_g)\} \ \mbox{where} \ f_{1}e_{1}+...f_{g}e_{g}=7. \] We only list possibilities where at least one of the entries in the ramification tuples is bigger than $1$. We finish by collecting the sets with equal arithmetic types such that their ramification tuples add to different values.\\

\begin{itemize}

\item $g=4.$\\

\begin{itemize}

\item[$\cdot$] $1+1+1+4$: \ $\{(1,1,1,1), (1,1,1,4)\}, \{(1,1,1,2), (1,1,1,2)\}.$

\item[$\cdot$] $1+1+2+3$: \ $\{(1,1,1,1), (1,1,2,3)\}, \{(1,1,1,3), (1,1,2,1)\}, \\ \{(1,1,1,2), (1,1,3,1)\}.$ 

\item[$\cdot$] $1+2+2+2$: \ $\{(1,1,1,1), (1,2,2,2)\}, \{(1,1,1,2), (1,2,2,1)\},  \\ \{(1,1,2,2), (1,2,1,1)\}.$ \\

\end{itemize}

In this case we have the pair \[\{(1,1,1,2), (1,1,1,2)\}, \{(1,1,1,2), (1,1,3,1)\}\] with ramification sums equal to $5$ and $6$ respectively and the pair \[\{(1,1,1,2), (1,1,1,2)\}, \{(1,1,1,2), (1,2,2,1)\}\] with the same pattern as the first pair.\\

\item $g=3.$ \\

\begin{itemize}

\item[$\cdot$] $1+1+5$. \ $\{(1,1,1), (1,1,5)\}$. 

\item[$\cdot$] $1+2+4$: \ $\{(1,1,1), (1,2,4)\}, \{(1,1,2), (1,2,2)\}, \{(1,1,2), (1,4,1)\}, \\ \{(1,1,4),(1,2,1)\},  \{(1,2,2), (1,1,2)\}.$

\item[$\cdot$] $1+3+3$: \ $\{(1,1,1), (1,3,3)\}, \{(1,1,3), (1,3,1)\}$.

\item[$\cdot$] $2+2+3$:  \ $\{(1,1,1), (1,2,3)\}, \{(1,1,2), (2,3,1)\}, \{(1,1,3), (2,2,1)\}, \\ \{(1,2,2),(3,1,1)\}, \{(1,2,3), (2,1,1)\}.$ \\

\end{itemize}

In this case we have the pair \[\{(1,1,2), (1,2,2)\}, \{(1,1,2), (1,4,1)\}\] with ramification sums equal to $5$ and $6$ respectively and the pair \[\{(1,1,2), (1,2,2)\}, \{(1,1,2), (2,3,1)\}\] with the same pattern as the first pair. Additionally we have \[\{(1,2,2), (1,1,2)\}, \{(1,2,2), (3,1,1)\}.\] Since non arithmetically solitary septic number fields are ${\rm PSL_{2}}(\F_{7})$ number fields (see for instance \cite{Klingen1} or \cite{Praeger}) it follows from Proposition \ref{Esporadicos} that a number field that is not arithmetically solitary can not have a prime with factorization type equal to  $\{(1,2,2), (3,1,1)\}$. \\

\item $g=2.$\\

\begin{itemize}

\item[$\cdot$] $1+6$: \ $\{(1,1), (1,6)\}, \{(1,2), (1,3)\}, \{(1,3), (1,2)\}$

\item[$\cdot$] $2+5$: \ $\{(1,1), (2,5)\}, \{(1,5), (2,1)\}, \{(1,2), (5,1)\}$

\item[$\cdot$] $3 +4$:  \ $\{(1,1), (3,4)\}, \{(1,2), (3,2)\}, \{(1,3), (4,1)\}, \\ \{(1,4), (3,1)\}, \{(2,3), (2,1)\}.$\\

\end{itemize}

In this case we have the pair \[\{(1,3), (1,2)\}, \{(1,3), (4,1)\}\] with ramification sums equal to $3$ and $5$ respectively and the trio \[\{(1,2), (1,3)\}, \{(1,2), (3,2)\}, \{(1,2), (5,1)\}.\] These last cases are covered thanks to Proposition \ref{Difi}. See also Remark \ref{ElRemarque}. \end{itemize}\end{proof}

Using that not arithmetically solitary septic fields have simple Galois group we narrow the possibilities of prime powers appearing in the discriminant of fields for which Question \ref{LaPregunta} could have a negative answer.

\begin{corollary}\label{ElCoro}
Let $K, K'$ be degree $7$ arithmetically equivalent number fields. Suppose $\ell$ is a rational prime which is not wildly ramified in either $K$ or $K'$, and let $v_{\ell}$ the usual $\ell$-adic valuation. If $\displaystyle e^{K}_{1}+...+e^{K}_{g} \neq e^{K'}_{1}+...+e^{K'}_{g}$ then $v_{\ell}({\rm disc}(O_{K})) \in \{2, 4\}$.%
\end{corollary}

\begin{proof}
Thanks to Theorem \ref{Cases} we see that the sum of the inertia degrees, at every prime $\ell$, in either field is either $3$,$4$ or $5$. Since for non wildly ramified primes $v_{\ell}({\rm disc}(O_{K}))=[K:\Q]-(f_{1}+...+f_{g})$ we see that $v_{\ell}({\rm disc}(O_{K})) \in \{2,3,4\}$. On the other hand a septic field with simple Galois group must have square discriminant since its Galois closure embeds in $A_{7}$, hence the result.
\end{proof}

\begin{remark}
From Theorem \ref{Cases} we see that the only primes that could give a negative answer to Question \ref{LaPregunta} and that are wildly ramified in both fields are $2$ and $3$. For instance, if $\ell=2$ a similar argument as the above shows that $v_{\ell}({\rm disc}(O_{K})) \in \{6,8\}$.
\end{remark}

Finally we show that Theorem \ref{Cases} is the best we can get in terms of Perlis and Stuart's question:

\begin{theorem}\label{Larespuesta}
For each tuple $\mathcal{F} \in \{(1,3), (1,1,2), (1,1,1,2)\}$ there are examples of pairs $(K,K')$ of non-isomorphic arithmetically equivalent number fields, and a prime $\ell$, with common arithmetic type $\mathcal{F}$ in $K$ and $K'$ such that \[e^{K}_{1}+...+e^{K}_{g} \neq e^{K'}_{1}+...+e^{K'}_{g}.\] 
\end{theorem}

\begin{proof} 
For $i=1,2$ consider the pairs of septic number fields ($K_{i}$, $K'_{i}$) defined by the ${\rm following}$ pairs of polynomials $(f_{i}, g_{i})$ respectively:\\
\begin{itemize}

\item $f_{1}:=x^7 - 3x^6 + 4x^5 - 5x^4 + 3x^3 - x^2 - 2x + 1$ and $g_{1}:=x^7 - x^5 - 2x^4 - 2x^3 + 2x^2 - x + 4$.\\

\item $f_{2}:=x^7 - 7x^5 - 14x^4 - 7x^3 - 7x + 2$ and $g_{2}:=x^7 - 14x^3 - 14x^2 + 7x + 22 $.\\

\end{itemize}

The first two fields have discriminant $2^{6}691^{2}$ and the second two have discriminant $2^{8}7^{8}$. A calculation, done in MAGMA, shows that $[K_{i}K'_{i}: \Q] \leq 28$. Since the fields have prime degree over $\Q$ it follows,  thanks to Theorem \ref{DisjointPrime} below, that $K_{i}$ and $K'_{i}$ are arithmetically equivalent. \\ 

\noindent For the given prime $\ell$, and the given field, the factorization type $\displaystyle \{(f_1,...,f_g), (e_1,...,e_g)\}$ is:

\begin{enumerate}

\item $\ell =2$ 
          \begin{itemize}
          
          \item[(a)]  $\displaystyle K_{1}; \ \{(1,3), (4,1)\}.$
          
          \item [(b)]  $\displaystyle K'_{1}; \ \{(1,3), (1,2)\}.$ \\
          
          \end{itemize}
          
\item $\ell =691$ 
          \begin{itemize}
          
          \item[(a)]  $\displaystyle K_{1}; \{(1,1,1,2), (1,1,1,2)\}.$
          
          \item [(b)]  $\displaystyle K'_{1}; \{(1,1,1,2), (2,1,2,1)\}.$\\
          
          \end{itemize}          

\item $\ell =2$ 
          \begin{itemize}
          
          \item[(a)] $\displaystyle K_{2}; \ \{(1,1,2), (1,4,1)\}$.
          
          \item [(b)] $\displaystyle K'_{2}; \ \{(1,1,2), (1,2,2)\}.$
          
          \end{itemize}

\end{enumerate}

\end{proof}

\begin{theorem}[Perlis \cite{Perlis2}]\label{DisjointPrime}
Let $K, K'$ be two number fields. Suppose that they have the same prime degree over $\Q$. Then $K$ and $K'$ are arithmetically equivalent if and only if they are linearly disjoint.
\end{theorem}

\noindent
{\footnotesize Department of Mathematics, Universidad Konrad Lorenz,
Bogot\'a, Colombia ({\tt gmantelia@gmail.com})}


\begin{thebibliography}{99}


\bibitem[CoMa]{Cornelissen} G. Cornellisen, M. Marcolli \textit{Quantum Statistical Mechanics, L-series and Anabelian
Geometry.}, arXiv:1009.0736.

\bibitem[CoKoVa]{Cornelissen2} G. Cornellisen, A. Kontogeorgis, L. Van der Zalm \textit{Arithmetic equivalence for function fields, the Goss zeta function and a generalisation.}, J. of Number Theory. \textbf{130} (2010), no. 4, 1000-1012.

\bibitem[Ga]{Gassmann} F. Gassman, \textit{Bemerkungen zu der vorstehenden Arbeit von Hurwitz.}, Math. Z. {\bf 25} (1926), 124-143.

\bibitem[JR]{Jones} J. Jones, D. Roberts, \textit{A data base of number fields.},  LMS Journal of Computation and Mathematics. {\bf 17, 1} (2014), 595-618.

\bibitem[Kl]{Klingen1} N. Klingen, \textit{Rigidity of Decomposition laws and number fields,} Journal of the Australian Mathematical Society. {\bf 51, 2} (1991), 171-186.



\bibitem[Kl1]{Klingen} N. Klingen, \textit{Arithmetical similarities. Prime decomposition and finite group theory,} Oxford Mathematical Monographs. Oxford Science Publications. The Clarendon Press, Oxford University Press, New York, 1998.


\bibitem[Ma]{MantillaS} G. Mantilla-Soler, \textit{Weak arithmetic equivalence.}, Canad. Math. Bull. \textbf{58} (2015), no. 1, 115-127.

\bibitem[P1]{Perlis1} R. Perlis, \textit{ On the equation $\zeta_{K}=\zeta_{K'}$}, J. of Number Theory. \textbf{9} (1977), 489-509.

\bibitem[P2]{Perlis} R. Perlis, \textit{ On the analytic determination of the trace form}, Canad. Math. Bull. \textbf{28} (4)
1985, 422-430.

\bibitem[P3]{Perlis2} R. Perlis, \textit{ A remark about zeta functions of number fields of prime degree}, J. Reine Angew. Math. \textbf{293/294} (1977), 435--436.

\bibitem[PS]{PerStu} R. Perlis, D. Stuart \textit{A new characterization of Arithmetic equivalence}, Journal of Number theory. \textbf{53} (1995), 300--308.

\bibitem[Pr]{Praeger} C. Praeger, \textit{Kronecker classes of field extensions of small degree,} Journal of the Australian Mathematical Society. {\bf 50} (1991), 297-315.

\bibitem[Pra]{Prasad} D. Prasad, \textit{A refined notion of arithmetically equivalent number fields, and curves with isomorphic jacobians},  arXiv:1409.3173
1985, 422-430.

\bibitem[Se]{Serre2} J.P. Serre, \textit{Local fields}, Graduate Texts in Mathematics, \textbf{67}. Springer-Verlag, New York-Berlin, 1979. viii+241 pp.

\end{thebibliography}
\end{document}